\documentclass[11pt]{article}
\usepackage{amsmath}
\usepackage{amssymb}

\newtheorem{theorem}{Theorem}

\newtheorem{lemma}{Lemma}
\newtheorem{proposition}{Proposition}
\newtheorem{remark}{Remark}

\newenvironment{proof}[1][Proof]{\noindent\textbf{#1.} }{\ \rule{0.5em}{0.5em}}
\begin{document}

\title{$q$-Sturm-Liouville theory and the corresponding eigenfunction expansions}
\author{Lazhar Dhaouadi\thanks{Ecole Pr\'{e}paratoire d'Ing\'{e}nieur, Bizerte,
Tunisia. E-mail lazhardhaouadi@yahoo.fr}}
\date{}
\maketitle

\begin{abstract}
The aim of this paper is to study the $q$-Schr\"{o}dinger operator
$$
L= q(x)-\Delta_q,
$$
where $q(x)$ is a given function of $x$  defined over
$\mathbb{R}_{q}^{+}=\{q^n,\quad n\in\mathbb Z\}$ and $\Delta_q$ is
the $q$-Laplace operator
$$
\Delta_{q}f(x)=\frac{1}{x^{2}}\left[
f(q^{-1}x)-\frac{1+q}{q}f(x)+\frac{1}{q}f(qx)\right].
$$
\end{abstract}

\section{Introduction}

After the spectral analysis in \cite{D} of the $q$-Laplace operator
also called the second-order $q$-difference operator
$$
\Delta_{q}f(x)=\frac{1}{x^{2}}\left[
f(q^{-1}x)-\frac{1+q}{q}f(x)+\frac{1}{q}f(qx)\right],
$$
it is natural to study the perturbed operator
$$
L= q(x)-\Delta_q.
$$

The eigenfunction expansion theory for the $q$-Sturm-Liouville
equation (singular case) presented in this paper is based on the
original works of Hermann Weyl in 1910 and of Edward Charles
Titchmarsh in 1941, concerning Sturm-Liouville theory and the
corresponding eigenfunction expansions. For this account the
essential results of Weyl concern the regular, limit-circle and
limit-point classifications of Sturm-Liouville differential
equations (singular case); the eigenfunction expansion theory from
Titchmarsh is based on classical function theory methods, in
particular complex function theory. For more information on the
classical theory, the reader can consult the references
\cite{H,S,T,W} .

\section{Basic definitions}

Consider $0<q<1$. In what follows, the standard conventional notations from
\cite{G}, will be used
$$
\mathbb{R}_{q}^{+}=\{q^{n},\quad n\in\mathbb{Z}\}
$$
$$
(a,q)_{0}=1,\quad(a,q)_{n}=\prod_{i=0}^{n-1}(1-aq^{i}).
$$
The $q$-schift operator is
$$
\Lambda_{q}f(x)=f(qx).
$$
Next we introduce two concepts of $q$-analysis: the $q$-derivative and the
$q$-integral. The $q$-derivative of a function $f$ is defined by
$$
D_{q}f(x)=\frac{f(x)-f(qx)}{(1-q)x},
$$
and $D_{q}f(0)=f^{\prime}(0)$, provided $f^{\prime}(0)$ exists. The
second order $q$-difference operator is
$$
{\small \Delta_{q}f(x)=\left[  \frac{1-q}{q}\right]  ^{2}\Lambda_{q}^{-1}%
D_{q}^{2}f(x)=\frac{1}{x^{2}}\left[  f(q^{-1}x)-\frac{1+q}{q}f(x)+\frac{1}%
{q}f(qx)\right]  .}
$$
The product rule for the $q$-derivative is
$$
D_{q}(fg)(x)=D_{q}f(x)g(x)+\Lambda_{q}f(x)D_{q}g(x).
$$
Jackson's $q$-integral (see \cite{J}) in the interval $[a,b]$ and in the
interval $[0,\infty\lbrack$ are defined by%

\begin{align*}
\int_{a}^{b}f(x)d_{q}x  &  =(1-q)\sum_{n=0}^{\infty}q^{n}[bf(bq^{n}
)-af(aq^{n})]\\
\int_{0}^{\infty}f(x)d_{q}x  &  =(1-q)\sum_{n=-\infty}^{\infty}q^{n}f(q^{n}).
\end{align*}
Also the rule of $q$-integration by parts is given by
$$
\int_{a}^{b}D_{q}f(x)g(x)d_{q}x=[f(b)g(b)-f(a)g(a)]-\int_{a}^{b}\Lambda
_{q}f(x)D_{q}g(x)d_{q}x.
$$
In the end we denote by $L^{2}(\mathbb{R}_q^+)$ the Hilbert space of
functions $f$ defined on $\mathbb R_q^+$ and satisfy
$$
\int_0^\infty |f(x)|^2d_qx<\infty.
$$

\section{$q$-Sturm-Liouville difference equation}

If $F$  satisfies the $q$-difference equation
\begin{equation}
\Delta_{q}f(x)+[\lambda-q(x)]f(x)=0, \label{1}
\end{equation}
and $G$ the same equation with $\lambda^{\prime}$ instead of
$\lambda$, then
\begin{align*}
&  (\lambda^{\prime}-\lambda)\int_{0}^{b}F(x)G(x)d_{q}x\\
&  =\int_{0}^{b}[F(x)\left\{  q(x)G(x)-\Delta_{q}G(x)\right\}  -G(x)\left\{
q(x)F(x)-\Delta_{q}F(x)\right\}  ]d_{q}x\\
&  =-\int_{0}^{b}\left\{
F(x)\Delta_{q}G(x)-G(x)\Delta_{q}F(x)\right\}
d_{q}x=W_{0}(F,G)-W_{b}(F,G),
\end{align*}
where $W_{x}$ is the $q$-Wronskian defined by
$$
W_{x}(F,G)=\frac{(1-q)^{2}}{q}\left[  F(x)\Lambda_{q}^{-1}D_{q}
G(x)-G(x)\Lambda_{q}^{-1}D_{q}F(x)\right]  .
$$
Note that the $q$-Wronskian defined here is slightly different from
the $q$-Wronskian introduced in \cite{Sw} and in \cite{D}. In some
cases we write $W\Big[F(x),G(x)\Big]$ instead of $W_x(F,G)$.

\bigskip
If $\lambda=\mu+i\nu$, $\lambda^{\prime}=\overline{\lambda}$ and
$G=\overline{F}$, this gives%
\begin{equation}
2\nu\int_{0}^{b}\left\vert F(x)\right\vert ^{2}d_{q}x=iW_{0}(F,\overline
{F})-iW_{b}(F,\overline{F}). \label{2}%
\end{equation}
Now let $\phi(x)=\phi(x,\lambda)$ and $\theta(x)=\theta(x,\lambda)$ be tows
solutions of (1) such that%
$$
\left\{
\begin{array}
[c]{c}%
\phi(0)=\frac{\sqrt{q}}{1-q}\sin\alpha,\quad\phi^{\prime}(0)=-\frac{\sqrt{q}}{1-q}\cos\alpha\\
\theta(0)=\frac{\sqrt{q}}{1-q}\cos\alpha,\quad\theta^{\prime}(0)=\frac{\sqrt{q}}{1-q}\sin\alpha
\end{array}
\right.  ,
$$
where $\alpha$ is real. Then it follows that

\begin{theorem}
For every value of $\lambda$ other than real values, there exist a
constante $m(\lambda)$ such that  (\ref{1}) has a solution
$$
\psi(x,\lambda)=\theta(x,\lambda)+m(\lambda)\phi(x,\lambda)
$$
belonging to $L^{2}(\mathbb{R}_{q}^{+})$.
\end{theorem}

\begin{proof}In fact
$$
W_{x}(\phi,\theta)=W_{0}(\phi,\theta)=\sin^{2}\alpha+\cos^{2}\alpha=1.
$$
and
$$
W_{0}(\phi,\overline{\phi})=W_{0}(\theta,\overline{\theta})=0.
$$
Therefore
$$
W_{0}(\theta+l\phi,\overline{\theta}+\overline{l}\overline{\phi}%
)=l-\overline{l}=2i\operatorname{Im}l.
$$
The general solution of (\ref{1}) is of the form $\theta(x)+l\phi(x).$
Consider those solutions which satisfy a real boundary condition at $x=b$, say%
$$
\{\theta(b)+l\phi(b)\}\cos\beta+\{\Lambda_{q}^{-1}D_{q}\theta(b)+l\Lambda
_{q}^{-1}D_{q}\phi(b)\}\sin\beta=0,
$$
where $\beta$ is real. This gives%
$$
l=l(\lambda)=\frac{\theta(b)\cot\beta+\Lambda_{q}^{-1}D_{q}\theta(b)}%
{\phi(b)\cot\beta+\Lambda_{q}^{-1}D_{q}\phi(b)}.
$$
For each $b$, as $\cot\beta$ varies, $l$ describes a circle in the complex
plane, say $C_{b}$. Replacing $\cot\beta$ by a complex variable $z$, we obtain%
$$
l=l(\lambda)=-\frac{\theta(b)z+\Lambda_{q}^{-1}D_{q}\theta(b)}{\phi
(b)z+\Lambda_{q}^{-1}D_{q}\phi(b)}.
$$
The centre of $C_{b}$ correspond to
$$
z=-\frac{\Lambda_{q}^{-1}D_{q}\overline{\phi}(b)}{\overline{\phi}(b)}%
=-\frac{W_{b}(\theta,\overline{\phi})}{W_{b}(\phi,\overline{\phi})}.
$$
Since $-\Lambda_{q}^{-1}D_{q}\theta(b)/\Lambda_{q}^{-1}D_{q}\phi(b)$ is on
$C_{b}$(for $z=0$) the radius $r_{b}$ of $C_{b}$ is%
\begin{equation}
r_{b}=\left\vert \frac{\Lambda_{q}^{-1}D_{q}\theta(b)}{\Lambda_{q}^{-1}%
D_{q}\phi(b)}-\frac{W_{b}(\theta,\overline{\phi})}{W_{b}(\phi,\overline{\phi
})}\right\vert =\left\vert \frac{W_{b}(\theta,\phi)}{W_{b}(\phi,\overline
{\phi})}\right\vert =\frac{1}{2\nu\int_{0}^{b}\left\vert \phi(x)\right\vert
^{2}d_{q}x}. \label{3}%
\end{equation}
Now $l$ is inside $\ C_b$ if $\operatorname{Im}z<0$, i.e. if%
$$
i\left\{  -\frac{l\Lambda_{q}^{-1}D_{q}\phi(b)+\Lambda_{q}^{-1}D_{q}\theta
(b)}{l\phi(b)+\theta(b)}+\frac{\overline{l}\Lambda_{q}^{-1}D_{q}\overline
{\phi}(b)+\Lambda_{q}^{-1}D_{q}\overline{\theta}(b)}{l\overline{\phi
}(b)+\overline{\theta}(b)}\right\}  >0,
$$
i.e. if
$$
i\left\{  \left\vert l\right\vert ^{2}W_{b}(\phi,\overline{\phi})+lW_{b}%
(\phi,\overline{\theta})+\overline{l}W_{b}(\theta,\overline{\phi}%
)+W_{b}(\theta,\overline{\theta})\right\}  >0,
$$
i.e. if
$$
iW_{b}(\theta+l\phi,\overline{\theta}+\overline{l}\overline{\phi})>0,
$$
i.e. if
$$
2\nu\int_{0}^{b}\left\vert \theta+l\phi\right\vert ^{2}d_{q}x<iW_{0}%
(\theta+l\phi,\overline{\theta}+\overline{l}\overline{\phi}).
$$
Hence $l$ is interior to $\ C_{b}$ if $\nu>0$, and
$$
\int_{0}^{b}\left\vert \theta+l\phi\right\vert ^{2}d_{q}x<-\frac
{\operatorname{Im}l}{\nu}.
$$
The same result is obtained if $\nu<0$. It follows that, if $l$ is interior to
$\ C_{b}$, and $0<b^{\prime}<b$, then%

$$
\int_{0}^{b^{\prime}}\left\vert \theta+l\phi\right\vert ^{2}d_{q}x<\int
_{0}^{b}\left\vert \theta+l\phi\right\vert ^{2}d_{q}x<-\frac{\operatorname{Im}%
l}{\nu}.
$$
Hence $l$ is also inside $C_{b^{\prime}}$. Hence $C_{b^{\prime}}$ includes
$C_{b}$ if $b^{\prime}<b$. It follows that, as $b\rightarrow\infty$, the
circle $C_{b}$ converge either to a limit-circle or to a limit-point.

\bigskip

If $m=m(\lambda)$ is the limit-point, or any point on the limit-circle,%
$$
\int_{0}^{b}\left\vert \theta+m\phi\right\vert ^{2}d_{q}x<-\frac
{\operatorname{Im}m}{\nu},
$$
for all values of $b$. Hence
$$
\int_{0}^{\infty}\left\vert \theta+m\phi\right\vert ^{2}d_{q}x<-\frac
{\operatorname{Im}m}{\nu}.
$$
This finish the proof.
\end{proof}

\begin{remark}
In the limit-circle case, $r_{b}$ tends to a positive limit as
$b\rightarrow\infty$. Hence, by (\ref{3}) the function $\phi$ is
$L^{2}(\mathbb{R}_{q}^{+})$. So in fact, in this case every solution
of (\ref{1}) belongs to $L^{2}(\mathbb{R}_{q}^{+}).$
\end{remark}

\section{The eigenfunctions}

On the circle $C_{b}$ (if $\nu>0$)
$$
\frac{1}{2}\left\vert l\right\vert ^{2}\int_{0}^{b}\left\vert \phi\right\vert
^{2}d_{q}x-\int_{0}^{b}\left\vert \theta\right\vert ^{2}d_{q}x\leq\int_{0}%
^{b}\left\vert \theta+l\phi\right\vert ^{2}d_{q}x<-\frac{\operatorname{Im}%
l}{\nu}\leq\frac{\left\vert l\right\vert }{v}.
$$
Solving for $\left\vert l\right\vert $ we obtain
$$
\left\vert l\right\vert \leq\frac{1}{\nu\int_{0}^{b}\left\vert \phi\right\vert
^{2}d_{q}x}+\left\{  \frac{2\int_{0}^{b}\left\vert \theta\right\vert ^{2}%
d_{q}x}{\int_{0}^{b}\left\vert \phi\right\vert ^{2}d_{q}x}+\frac{1}{\left(
\nu\int_{0}^{b}\left\vert \phi\right\vert ^{2}d_{q}x\right)  ^{2}}\right\}
^{\frac{1}{2}}.
$$
Since the above right-hand side is $O(\frac{1}{\nu})$. Hence as $\nu
\rightarrow0$, for any fixed $b$, it also follows that $m(\lambda)=O(\frac
{1}{\nu})$. Hence, if $m(\lambda)$ has poles on the real axis, they are all
simple. In this paper we assume that $m(\lambda)$ form a single analytic
function, whose only singularities are poles on the real axis. Let them be
$\lambda_{0},\lambda_{1},\ldots,$ and let the residues be $r_{0},r_{1},\ldots$.

\bigskip

\begin{lemma}
For any fixed complex $\lambda$ and $\lambda^{\prime}$
$$
\lim_{x\rightarrow\infty}W_{x}\{\psi(.,\lambda),\psi(.,\lambda^{\prime})\}=0.
$$
\end{lemma}

\begin{proof}
Since
$$
W_{b}\{\theta(x,\lambda)+l(\lambda)\phi(x,\lambda),\theta(x,\lambda^{\prime
})+l(\lambda^{\prime})\phi(x,\lambda^{\prime})\}=0,
$$
i.e.
$$
W_{b}[\psi(x,\lambda)+\{l(\lambda)-m(\lambda)\}\phi(x,\lambda),\psi
(x,\lambda^{\prime})+\{l(\lambda^{\prime})-m(\lambda^{\prime})\}\phi
(x,\lambda^{\prime})]=0,
$$
i.e.
\begin{align*}
&  W_{b}\{\psi(x,\lambda),\psi(x,\lambda^{\prime})\}+\{l(\lambda
)-m(\lambda)\}W_{b}\{\phi(x,\lambda),\psi(x,\lambda^{\prime})\}\\
&  +\{l(\lambda^{\prime})-m(\lambda^{\prime})\}W_{b}\{\psi(x,\lambda
),\phi(x,\lambda^{\prime})\}\\
&  +\{l(\lambda)-m(\lambda)\}\{l(\lambda^{\prime})-m(\lambda^{\prime}%
)\}W_{b}\{\phi(x,\lambda),\phi(x,\lambda^{\prime})\}\\
&  =0.
\end{align*}
Now
\begin{align*}
W_{b}\{\phi(x,\lambda),\psi(x,\lambda^{\prime})\}  &  =(\lambda-\lambda
^{\prime})\int_{0}^{b}\phi(x,\lambda)\psi(x,\lambda^{\prime})d_{q}%
x+W_{0}\{\phi(x,\lambda),\psi(x,\lambda^{\prime})\}\\
&  =O\left\{  \int_{0}^{b}\left\vert \phi(x,\lambda)\right\vert ^{2}%
d_{q}x\right\}  ^{\frac{1}{2}}+O(1),
\end{align*}
as $b\rightarrow\infty$, $\lambda$ and $\lambda^{\prime}$ being
fixed. In the limit-point case
$$
\left\vert l(\lambda)-m(\lambda)\right\vert \leq2r_{b}=\left\{  \nu\int
_{0}^{b}\left\vert \phi(x,\lambda)\right\vert ^{2}d_{q}x\right\}  ^{-1},
$$
so that
$$
\lim_{b\rightarrow\infty}\{l(\lambda)-m(\lambda)\}W_{b}\{\phi(x,\lambda
),\psi(x,\lambda^{\prime})\}=0.
$$
This also holds in the limit-circle case, if $l(\lambda)\rightarrow
m(\lambda)$, since then $\int_{0}^{b}\left\vert \phi(x,\lambda)\right\vert
^{2}d_{q}x$ is bounded. Similar arguments apply to the other terms.
\end{proof}

\begin{lemma}
Let $\{f_n\}$ be a sequence of functions which converges in mean
square to $f$ over any finite interval, while
$$
\int_{0}^{\infty}\left\vert f_{n}(x)\right\vert ^{2}d_{q}x\leq K
$$
for all $n$. Then $f$ is $L^{2}(\mathbb{R}_{q}^{+})$, and if $g$
belongs to $L^{2}(\mathbb{R}_{q}^{+})$,
$$
\lim_{n\rightarrow\infty}\int_{0}^{\infty}f_{n}(x)g(x)d_{q}x=\int_{0}^{\infty
}f(x)g(x)d_{q}x.
$$
\end{lemma}

\begin{proof}
We have
$$
\int_{0}^{X}\left\vert f(x)\right\vert ^{2}d_{q}x=\lim_{n\rightarrow\infty
}\int_{0}^{X}\left\vert f_{n}(x)\right\vert ^{2}d_{q}x\leq K
$$
for every $X$, so that $f$ is $L^{2}(\mathbb{R}_{q}^{+})$. Now%
\begin{align*}
\left\vert \int_{0}^{\infty}(f-f_{n})gd_{q}x\right\vert  &  \leq\left\vert
\int_{0}^{X}\right\vert +\left\vert \int_{X}^{\infty}\right\vert \\
&  \leq\left\{  \int_{0}^{X}\left\vert f-f_{n}\right\vert ^{2}d_{q}x\int
_{0}^{\infty}\left\vert g\right\vert ^{2}d_{q}x\right\}  ^{\frac{1}{2}%
}+\left\{  \int_{0}^{\infty}\left\vert f-f_{n}\right\vert ^{2}d_{q}x\int
_{X}^{\infty}\left\vert g\right\vert ^{2}d_{q}x\right\}  ^{\frac{1}{2}}.
\end{align*}
This finish the proof.
\end{proof}

\begin{proposition}
The functions
$$
\psi_{n}(x)=\sqrt{r_{n}}\phi(x,\lambda_{n})
$$
form a normal orthogonal set of $L^{2}(\mathbb{R}_{q}^{+}).$
\end{proposition}

\begin{proof}
By (2) if $\lambda$ and $\lambda^{\prime}$ are not real, we have%
\begin{align*}
(\lambda^{\prime}-\lambda)\int_{0}^{b}\psi(x,\lambda)\psi(x,\lambda^{\prime
})d_{q}x  &  =W_{0}\{\psi(x,\lambda),\psi(x,\lambda^{\prime})\}-W_{b}%
\{\psi(x,\lambda),\psi(x,\lambda^{\prime})\}\\
&
=m(\lambda)-m(\lambda^{\prime})-W_{b}\{\psi(x,\lambda),\psi(x,\lambda
^{\prime})\}.
\end{align*}
By Lemma 1, the second term in the right tends to zero as
$b\rightarrow\infty$. Hence
\begin{equation}
\int_{0}^{\infty}\psi(x,\lambda)\psi(x,\lambda^{\prime})d_{q}x=\frac
{m(\lambda)-m(\lambda^{\prime})}{\lambda^{\prime}-\lambda}.
\label{10}
\end{equation}
In particular, taking $\lambda^{\prime}=\overline{\lambda}$, we obtain%
$$
\int_{0}^{\infty}\left\vert \psi(x,\lambda)\right\vert ^{2}d_{q}%
x=-\frac{\operatorname{Im}\left\{  m(\lambda)\right\}  }{\nu}.
$$
Now \ let $\lambda_{n}$ be an eigenvalue, and let $\lambda^{\prime}%
=\lambda_{n}+i\nu,$ $\nu\rightarrow0$. Then for any fixed $X$,%
\begin{align*}
&  \int_{0}^{X}\left\vert \nu\psi(x,\lambda^{\prime})+ir_{n}\phi(x,\lambda
_{n})\right\vert ^{2}d_{q}x\\
&  =\int_{0}^{X}\left\vert \nu\theta(x,\lambda^{\prime})+\{\nu m(\lambda
^{\prime})+ir_{n}\}\phi(x,\lambda^{\prime})-ir_{n}\{\phi(x,\lambda^{\prime
})-\phi(x,\lambda_{n})\right\vert ^{2}d_{q}x\rightarrow0.
\end{align*}
Also, by
$$
\int_{0}^{\infty}\left\vert \nu\psi(x,\lambda^{\prime})\right\vert ^{2}%
d_{q}x\leq\left\vert \nu m(\lambda^{\prime})\right\vert =O(1)
$$
as $\nu\rightarrow0$, since the pole of $m(\lambda^{\prime})$ at $\lambda_{n}$
is simple. On multiplying (\ref{10}) by $i\nu/r_{m}$, making $\nu\rightarrow
0$, and using Lemma 2 we see that $\phi(x,\lambda_{n})$ is $L^{2}%
(\mathbb{R}_{q}^{+})$, and%
\begin{equation}
\int_{0}^{\infty}\psi(x,\lambda)\phi(x,\lambda_{n})d_{q}x=\frac{1}%
{\lambda-\lambda_{n}}. \label{4}%
\end{equation}
If $\lambda$ tends to a different eigenvalue $\lambda_{m}$, on multiplying
(\ref{4}) by $i\nu/r_{m}$ and making $\nu\rightarrow0$, we obtain%
$$
\int_{0}^{\infty}\phi(x,\lambda_{m})\phi(x,\lambda_{n})d_{q}x=0.
$$
If $\lambda$ tends to the same eigenvalue $\lambda_{n}$, it follows
similarly that
$$
\int_{0}^{\infty}\left\{  \phi(x,\lambda_{n})\right\}
^{2}d_{q}x=\frac {1}{r_{n}},
$$
which leads to the result.
\end{proof}

\section{Series expansions}

Let $f$ be $L^{2}(\mathbb{R}_{q}^{+})$ and let
$$
\Phi(x,\lambda)=\psi(x,\lambda)\int_{0}^{x}\phi(y,\lambda)f(y)d_{q}%
y+\phi(x,\lambda)\int_{x}^{\infty}\psi(y,\lambda)f(y)d_{q}y,
$$
where $\phi$ and $\psi$ are the functions defined above. In the
following we denote by $c_n$ the nth Fourier coefficient of the
function $f$
$$
c_n=\int_{0}^{\infty}\psi_n(y)f(y)d_qy.
$$

\begin{proposition}
The function $$\lambda\mapsto\Phi(x,\lambda)$$ has a simple pole at
$\lambda_{n}$, its residue is $ c_{n}\psi_{n}(x)$.
\end{proposition}

\begin{proof}We have
$$
D_{q}\Phi(x,\lambda)=D_{q}\psi(x,\lambda)\int_{0}^{qx}\phi(y,\lambda
)f(y)d_{q}y+D_{q}\phi(x,\lambda)\int_{qx}^{\infty}\psi(y,\lambda)f(y)d_{q}y
$$
and
\begin{align*}
D_{q}^{2}\Phi(x,\lambda)  &  =D_{q}^{2}\psi(x,\lambda)\int_{0}^{q^{2}x}%
\phi(y,\lambda)f(y)d_{q}y+D_{q}^{2}\phi(x,\lambda)\int_{q^{2}x}^{\infty}%
\psi(y,\lambda)f(y)d_{q}y\\
&  +qD_{q}\psi(x,\lambda)\phi(qx,\lambda)f(qx)-qD_{q}\phi(x,\lambda
)\psi(qx,\lambda)f(qx).
\end{align*}
Therefore
\begin{align*}
\Delta_{q}\Phi(x,\lambda)  &  ={\small \left[  \frac{1-q}{q}\right]
^{2}\Lambda_{q}^{-1}D_{q}^{2}}\Phi(x,\lambda)\\
&  =\Delta_{q}\psi(x,\lambda)\int_{0}^{qx}\phi(y,\lambda)f(y)d_{q}y+\Delta
_{q}\phi(x,\lambda)\int_{qx}^{\infty}\psi(y,\lambda)f(y)d_{q}y\\
&  +\frac{(1-q)^{2}}{q}\left[  {\small \Lambda_{q}^{-1}}D_{q}\psi
(x,\lambda)\phi(x,\lambda)-{\small \Lambda_{q}^{-1}}D_{q}\phi(x,\lambda
)\psi(x,\lambda)\right]  f(x)\\
&  =[q(x)-\lambda]\Phi(x,\lambda)+W_{x}(\theta,\phi)f(x)=[q(x)-\lambda
]\Phi(x,\lambda)+f(x).
\end{align*}
The function $\Phi$ satisfies the boundary condition%
\begin{equation}
\Phi(0,\lambda)\cos\alpha+\Phi^{\prime}(0,\lambda)\sin\alpha=0.\label{*}
\end{equation}
If $\Phi_{X}(x,\lambda)$ is the corresponding function with $f(y)=0$
for $y>X$, then
\begin{align*}
\Phi_{X}(x,\lambda)  &  =\theta(x,\lambda)\int_{0}^{x}\phi(y,\lambda
)f(y)d_{q}y+\phi(x,\lambda)\int_{x}^{X}\phi(y,\lambda)f(y)d_{q}y\\
&  +m(\lambda)\phi(x,\lambda)\int_{0}^{X}\phi(y,\lambda)f(y)d_{q}y.
\end{align*}
This is clearly regular everywhere except at
$\lambda=\lambda_{0},\lambda _{1},\ldots,$ where it hase simple
poles with residues
$$
r_{n}\phi(x,\lambda_{n})\int_{0}^{X}\phi(y,\lambda_{n})f(y)d_{q}y.
$$
Hence making $X\rightarrow\infty$ we find that $\Phi(x,\lambda)$ has
a simple pole at $\lambda_{n}$, its residue there being limit of the
residue of $\Phi_{X}(x,\lambda)$, i.e.
$$
r_{n}\phi(x,\lambda_{n})\int_{0}^{\infty}\phi(y,\lambda_{n})f(y)d_{q}%
y=\psi_{n}(x)\int_{0}^{\infty}\psi_{n}(y)f(y)d_{q}y=c_{n}\psi_{n}(x).
$$
This finish the proof.
\end{proof}

\begin{lemma}
If $f$ is any function of $L^{2}(\mathbb{R}_{q}^{+})$ then
\[
\int_{0}^{\infty}\left\vert \Phi(x,\lambda)\right\vert ^{2}d_{q}x\leq\frac
{1}{\nu^{2}}\int_{0}^{\infty}\left\vert f(x)\right\vert ^{2}d_{q}x.
\]
\end{lemma}

\begin{proof}
Suppose first that $f(x)=0$ for $x\geq X$. Then the condition of
self-adjointness
$$
\int_{0}^{\infty}\Phi(x,\lambda)L\Phi(x,\lambda^{\prime})d_{q}x=\int
_{0}^{\infty}\Phi(x,\lambda^{\prime})L\Phi(x,\lambda)d_{q}x,
$$
is satisfied. Indeed
\begin{align*}
&  \int_{0}^{\xi}\left\{  \Phi(x,\lambda)L\Phi(x,\lambda^{\prime}%
)-\Phi(x,\lambda^{\prime})L\Phi(x,\lambda)\right\}  d_{q}x\\
&  =-\int_{0}^{\xi}\left\{  \Phi(x,\lambda)\Delta_{q}\Phi(x,\lambda^{\prime
})-\Phi(x,\lambda^{\prime})\Delta_{q}\Phi(x,\lambda)\right\}  d_{q}x\\
&  =-W[\Phi(x,\lambda),\Phi(x,\lambda^{\prime})]_{0}^{\xi}.
\end{align*}
By $(\ref{*})$ the integrated term vanishes at $x=0$. The integrated
term at $x=\xi$ tend to $0$ as $\xi\rightarrow\infty$. Since, if
$x>X$ we have
$$
\Phi(x,\lambda)=\psi(x,\lambda)\int_{0}^{X}\phi(y,\lambda)f(y)d_{q}y,
$$
then the result follows from Lemma 1.

\bigskip
Putting $\lambda^{\prime}=\overline {\lambda}$ we obtain
$$
\int_{0}^{\infty}\Phi(x,\lambda)\{\overline{\lambda}\Phi(x,\overline{\lambda
})-f(x)\}d_{q}x=\int_{0}^{\infty}\Phi(x,\overline{\lambda})\{\lambda
\Phi(x,\lambda)-f(x)\}d_{q}x,
$$
i.e.%
$$
(\lambda-\overline{\lambda})\int_{0}^{\infty}\left\vert \Phi(x,\lambda
)\right\vert ^{2}d_{q}x=\int_{0}^{\infty}\{\Phi(x,\overline{\lambda}%
)-\Phi(x,\lambda)\}f(x)d_{q}x.
$$
Hence, if $\lambda=\mu+i\nu,\nu>0,$%
$$
2\nu\int_{0}^{\infty}\left\vert \Phi(x,\lambda)\right\vert ^{2}d_{q}x\leq
2\int_{0}^{\infty}\left\vert \Phi(x,\lambda)f(x)\right\vert d_{q}%
x\leq2\left\{  \int_{0}^{\infty}\left\vert \Phi(x,\lambda)\right\vert
^{2}d_{q}x\int_{0}^{\infty}\left\vert f(x)\right\vert ^{2}d_{q}x\right\}  .
$$
This prove the result in the restricted case. If now $f$ is any
function of
$L^{2}(\mathbb{R}_{q}^{+})$, for a fixed $X^{\prime}$ we have%
\begin{align*}
\int_{0}^{X^{\prime}}\left\vert \Phi_{X}(x,\lambda)\right\vert ^{2}d_{q}x  &
\leq\int_{0}^{\infty}\left\vert \Phi_{X}(x,\lambda)\right\vert ^{2}d_{q}x\\
&  \leq\frac{1}{\nu^{2}}\int_{0}^{\infty}\left\vert f_{X}(x)\right\vert
^{2}d_{q}x=\frac{1}{\nu^{2}}\int_{0}^{X}\left\vert f(x)\right\vert ^{2}%
d_{q}x\leq\frac{1}{\nu^{2}}\int_{0}^{\infty}\left\vert f(x)\right\vert
^{2}d_{q}x.
\end{align*}
The result therefore follows on making first $X\rightarrow\infty$, then
$X^{\prime}\rightarrow\infty$.
\end{proof}

\bigskip
We denote by $\mathcal{L}_{q,2}$ the subspace of
$L^{2}(\mathbb{R}_{q}^{+})$ of functions which satisfy

\begin{itemize}
\item The function $Lf$ be $L^{2}(\mathbb{R}_{q}^{+})$.

\item $f^{\prime}(0)\cos\alpha-f(0)\sin\alpha=0$.

\item $\lim_{x\rightarrow\infty}W\{\psi(x,\lambda),f(x)\}=0$, for every
non-real $\lambda.$
\end{itemize}

\bigskip
\begin{lemma}
If $f$ belongs to $\mathcal{L}_{q,2}$ and $\Phi(x,\lambda)$ defined
above is also denoted by $\Phi(x,\lambda,f)$ then
\begin{equation}
\Phi(x,\lambda,f)=\frac{1}{\lambda}\{f(x)+\Phi(x,\lambda,Lf)\}.
\label{5}
\end{equation}
\end{lemma}

\begin{proof}
We have
\begin{align*}
\int_{0}^{x}\phi(y,\lambda)f(y)d_{q}y  &  =\frac{1}{\lambda}\int_{0}%
^{x}\{q(y)\phi(y)-\Delta_{q}\phi(y)\}f(y)d_{q}y\\
&  =\frac{1}{\lambda}W[\phi(y),f(y)]_{0}^{x}+\frac{1}{\lambda}\int_{0}%
^{x}\{q(y)f(y)-\Delta_{q}f(y)\}\phi(y)d_{q}y.
\end{align*}
Similarly
\begin{align*}
\int_{x}^{\infty}\phi(y,\lambda)f(y)d_{q}y  &  =\frac{1}{\lambda}\int
_{x}^{\infty}\{q(y)\phi(y)-\Delta_{q}\phi(y)\}f(y)d_{q}y\\
&  =\frac{1}{\lambda}W[\psi(y),f(y)]_{x}^{\infty}+\frac{1}{\lambda}\int
_{x}^{\infty}\{q(y)f(y)-\Delta_{q}f(y)\}\psi(y)d_{q}y,
\end{align*}
and the integrated term vanishes at the upper limit.
\end{proof}

\bigskip
The Green's function $G(x,y,\lambda)$ is defined by%
$$
G(x,y,\lambda)=\left\{
\begin{array}
[c]{c}
-\psi(x,\lambda)\phi(y,\lambda)\text{, if }y\leq x\\
-\psi(y,\lambda)\phi(x,\lambda)\text{, if }y>x
\end{array}
\right.  ,
$$
then
$$
\Phi(x,\lambda,f)=-\int_{0}^{\infty}G(x,y,\lambda)f(y)d_{q}y.
$$
Therefore
\begin{equation}
f(x)=\int_{0}^{\infty}G(x,y,\lambda)\{Lf(y)-\lambda f(y)\}d_{q}y. \label{6}%
\end{equation}

\begin{lemma}
Let $F(\lambda)$ be an analytic function of $\lambda=\mu+i\nu$,
regular for $-r\leq \mu\leq r,-r\leq \nu\leq r$, and let
$$
|F(\lambda)|\leq\frac{M}{|\nu|},
$$
in this square. Then
$$
|F(\lambda)|\leq\frac{3M}{r}.
$$
\end{lemma}

\begin{proof}
Let
$$
G(\lambda)=(\lambda^2-r^2)F(\lambda).
$$
On the upper and lower sides of the square
$$
|G(\lambda)|\leq (|\lambda|^2+r^2)\frac{M}{r}\leq3rM.
$$
On the left-hand and right-hand sides
$$
|G(\lambda)|\leq |\nu|(|\lambda|+r)\frac{M}{|\nu|}\leq 3rM.
$$
Hence $|G(\lambda)|\leq 3rM$ throughout the square. Hence on the
imaginary axis
$$
|F(\lambda)|\leq
\frac{3rM}{|\lambda^2-r^2|}=\frac{3rM}{\nu^2+r^2}\leq\frac{3M}{r}.
$$
This finish the proof.
\end{proof}

\begin{theorem}
Let $f$ be a function belonging in $\mathcal{L}_{q,2}$ then%
\[
\int_{0}^{\infty}\{f(x)\}^{2}d_{q}x=\sum_{n=0}^{\infty}c_{n}^{2},
\]
\end{theorem}

\begin{proof}
Suppose that $f$ satisfies the conditions of the above theorem, and
also
that $f(x)=0$ for sufficient large values of $x$. Let%
$$
\Psi(\lambda)=\int_{0}^{\infty}f(x)\Phi(x,\lambda)d_{q}x,
$$
then $\Psi(\lambda)$ is regular except for simple poles at the points
$\lambda_{n}$, where it hase residues%
$$
c_{n}\int_{0}^{\infty}\psi_{n}(x)f(x)d_{q}x=c_{n}^{2}.
$$
By (\ref{5})
\begin{equation}
\Psi(\lambda)=\frac{1}{\lambda}\int_{0}^{\infty}\left\{  f(x)\right\}
^{2}d_{q}x+\frac{1}{\lambda}\int_{0}^{\infty}\Phi(x,\lambda,Lf)f(x)d_{q}x,
\label{11}%
\end{equation}
and the last term is
$$
O\left\{  \frac{1}{\left\vert \lambda\right\vert }\left[  \int_{0}^{\infty
}\left\vert \Phi(x,\lambda,Lf)\right\vert ^{2}d_{q}x\int_{0}^{\infty}\left\{
f(x)\right\}  ^{2}d_{q}x\right]  ^{\frac{1}{2}}\right\}  =O\left(  \frac
{1}{\left\vert \lambda\nu\right\vert }\right)  ,
$$
by Lemma 3, applied to \ $Lf$.

\bigskip
Let $C(R)$ denote the contour formed by the segments of lines $(R-i,R+i)$ and
$(-R-i,-R+i)$, joined by semicircles of radius $R$ and centres $\pm i$. Then%
$$
\int_{C(R)}\Psi(\lambda)d\lambda=\sum_{-R<\lambda_{n}<R}c_{n}^{2}.
$$
On the part of the upper semicircle in the first quadrant, we have%
$$
\lambda=i+R\text{ }e^{i\theta}\text{ \ \ \ }(0\leq\theta\leq\frac{\pi}{2}).
$$
Hence the last term in (\ref{11}), integrated round this quadrant, gives%
$$
O\left\{  \int_{0}^{\frac{\pi}{2}}\frac{R}{R(1+R\sin\theta}d\theta\right\}
=O\left\{  \int_{0}^{\frac{1}{R}}d\theta\right\}  +O\left\{  \int_{\frac{1}%
{R}}^{\frac{\pi}{2}}\frac{d\theta}{R\theta}\right\}  =O\left(  \frac{1}%
{R}\right)  +O\left(  \frac{\log R}{R}\right)  =o(1).
$$
A similar argument applies to the other quadrants. Hence the integral of
$\Psi(\lambda)$ round each semicircle tends to
$$
\pi i\int_{0}^{\infty}\{f(x)\}^{2}d_{q}x
$$
as $R\rightarrow\infty$. To prove the theorem for the class of functions
considered, it is therefore sufficient to prove that%
$$
\lim_{R\rightarrow\infty}\int_{R-i}^{R+i}\Psi(\lambda)d\lambda=0,
$$
and a similar result with $-R$ in place of $R$.

\bigskip

Let
$$
\chi(\lambda)=\Psi(\lambda)-\sum_{R-1\leq\lambda_{n}\leq R+1}\frac{c_{n}^{2}%
}{\lambda-\lambda_{n}}.
$$
Then $\chi(\lambda)$ is regular for $R-1\leq\lambda\leq R+1$, and%
$$
\left\vert \chi(\lambda)\right\vert <\frac{K}{\left\vert \lambda\nu\right\vert
}+\frac{1}{\nu}\sum_{R-1\leq\lambda_{n}\leq R+1}c_{n}^{2}\leq\frac
{\epsilon(R)}{\left\vert \nu\right\vert },
$$
where $\epsilon(R)\rightarrow0$ as $R\rightarrow\infty$. Hence, by Lemma 5%
$$
\left\vert \chi(\lambda)\right\vert \leq3\epsilon(R)
$$
on the segment $(R-i,R+i)$. Hence%
$$
\lim_{R\rightarrow\infty}\int_{R-i}^{R+i}\chi(\lambda)d\lambda=0.
$$
Also
$$
\int_{R-i}^{R+i}\frac{d\lambda}{\lambda-\lambda_{n}}d\lambda=O(1).
$$
Hence
$$
\int_{R-i}^{R+i}\sum_{R-1\leq\lambda_{n}\leq R+1}\frac{c_{n}^{2}}%
{\lambda-\lambda_{n}}d\lambda=O\left\{  \sum_{R-1\leq\lambda_{n}\leq R+1}%
c_{n}^{2}\right\}  =o(1).
$$
This prove the theorem for the special class of functions. The
theorem can be extended to all functions of integrable square.
\end{proof}

\begin{remark}
It also follows that, if $g$ is another function of
$L^{2}(\mathbb{R}_{q}^{+})$, with Fourier coefficients $d_{n}$, then
\begin{equation}
\int_{0}^{\infty}f(x)g(x)d_{q}x=\sum_{n=0}^{\infty}c_{n}d_{n}.
\label{7}
\end{equation}
\end{remark}

\begin{theorem}
Let $f$ be a function belonging in $\mathcal{L}_{q,2}$ then%
$$
f(x)=\sum_{n=0}^{\infty}c_{n}\psi_{n}(x),\quad\forall
x\in\mathbb{R}_{q}^{+}.
$$
\end{theorem}

\begin{proof}
From the formula
$$
(\lambda_{n}-\lambda)\int_{0}^{b}\psi(x,\lambda)\phi(x,\lambda_{n}%
)d_{q}x=W_{0}\{\psi(x,\lambda),\phi(x,\lambda_{n})\}-W_{b}\{\psi
(x,\lambda),\phi(x,\lambda_{n})\},
$$
and (5), it follows that the second term on the right-hand side
tends to $0$ as $b\rightarrow\infty$. Hence we may take
$f(x)=\psi_{n}(x)$ in (\ref{6}), which gives
$$
\int_{0}^{\infty}G(x,y,\lambda)\psi_{n}(x)d_{q}y=\frac{\psi_{n}(x)}%
{\lambda_{n}-\lambda}.
$$
Also
\begin{align*}
&  \int_{0}^{\infty}\left\vert G(x,y,\lambda)\right\vert ^{2}d_{q}y\\
&  =\left\vert \psi(x,\lambda)\right\vert ^{2}\int_{0}^{x}\left\vert
\phi(x,\lambda)\right\vert ^{2}d_{q}y+\left\vert
\phi(x,\lambda)\right\vert ^{2}\int_{x}^{\infty}\left\vert
\psi(x,\lambda)\right\vert ^{2}d_{q}y\leq K,
\end{align*}
say, for $x$ in a finite interval, and $\lambda$ not real. Hence by
the Bessel inequality
\begin{equation}
\sum_{n=0}^{\infty}\left\vert \frac{\psi_{n}(x)}{\lambda_{n}-\lambda
}\right\vert ^{2}\leq K. \label{8}%
\end{equation}
Now let $f$ satisfies the conditions of Theorem 3, and let%
$$
g(x)=Lf(x)-\lambda f(x).
$$
If $\lambda$ is not real,
$$
\int_{0}^{\infty}[\psi(y,\lambda)Lf(y)-f(y)L\{\psi(y,\lambda)\}]d_{q}%
y=W[f,\psi]_{0}^{\infty}=f^{\prime}(0)\cos\alpha-f(0)\sin\alpha=0,
$$
and $L\psi=\lambda\psi.$ Put $\lambda=\lambda_{n}+i\nu$, multiply by $\nu$,
and make $\nu\rightarrow0$. Using Lemma 3, we obtain%
$$
\int_{0}^{\infty}\psi_{n}(y)Lf(y)d_{q}y=\lambda_{n}\int_{0}^{\infty}\psi
_{n}(y)f(y)d_{q}y=\lambda_{n}c_{n}.
$$
Hence, if $d_{n}$ is the Fourier coefficient of $g(x)$,
$$
d_{n}=(\lambda_{n}-\lambda)c_{n},
$$
then
\begin{equation}
\sum_{n=0}^{\infty}\left\vert \lambda_{n}-\lambda\right\vert
^{2}c_{n}^{2}, \label{9}
\end{equation}
is convergente. From (\ref{6}) and (\ref{7}) with $f(y)$ replaced by
$G(x,y,\lambda)$, it follows that%
$$
f(x)=\int_{0}^{\infty}G(x,y,\lambda)g(y)d_{q}y=\sum_{n=0}^{\infty}\frac
{\psi_{n}(x)}{\lambda_{n}-\lambda}(\lambda_{n}-\lambda)c_{n}=\sum
_{n=0}^{\infty}c_{n}\psi_{n}(x),
$$
the required result. The absolute and uniform convergence of the series
follows from (\ref{8}) and the convergence of (\ref{9}).
\end{proof}

\end{document}